\documentclass[11pt, reqno, a4paper, final]{amsart}

\usepackage[applemac]{inputenc}
\usepackage[english]{babel}
\usepackage[T1]{fontenc}
\usepackage{amsmath,amsfonts,amssymb,amsthm,amscd}
\usepackage{stmaryrd}
\usepackage{eucal}
\usepackage{array}
\usepackage{mathtools}
\usepackage{color}
\usepackage{hyperref}
\usepackage{a4wide}										 	
\usepackage{mathrsfs}									  
\usepackage[all]{xy}										
\usepackage{url}												
\usepackage{verbatim,epsfig}
\usepackage[notref,notcite,draft]{showkeys}   
\usepackage{xfrac} 
\usepackage{setspace}
\usepackage{a4wide}

\theoremstyle{plain}

\newtheorem{thm}{Theorem}[section] 
 
\newtheorem{cor}[thm]{Corollary}

\newtheorem{lem}[thm]{Lemma}
\newtheorem{lemma}[thm]{Lemma}
\newtheorem{prop-defi}[thm]{Definition \& Proposition}
\newtheorem{prop}[thm]{Proposition}

\newtheorem*{thm*}{Theorem}
\newtheorem*{prop*}{Proposition}
\newtheorem*{cor*}{Corollary}

\theoremstyle{definition}
\newtheorem{defi}[thm]{Definition}

\newtheorem{rem}[thm]{Remark}

\newtheorem*{claim*}{Claim}

\numberwithin{equation}{section}

\usepackage{thmtools}
\declaretheorem[style=theorem,name={Theorem}]{theoremletter}

\newcommand{\NN}{{\mathbb N}}
\newcommand{\ZZ}{{\mathbb Z}}
\newcommand{\BB}{{\mathbb B}}

\newcommand{\QQ}{{\mathbb Q}}
\newcommand{\RR}{{\mathbb R}}
\newcommand{\CC}{{\mathbb C}}

\newcommand{\MM}{{\mathbb M}}

\newcommand{\C}{{\mathscr C}}

\renewcommand{\max}{{\operatorname{max}}}

\newcommand{\ip}[2]{\left\langle {#1}\hspace{0.05cm}, \hspace{0.05cm}{#2}\right \rangle}

\newcommand{\varps}{{\varepsilon}}

\newcommand{\tr}{{\operatorname{tr}}}

\newcommand{\alg}{{\operatorname{alg}}}

\newcommand{\Tr}{\operatorname{Tr}}

\newcommand{\Pol}{{\operatorname{Pol}}}

\renewcommand{\min}{{\operatorname{min}}}

\renewcommand{\leq}{\leqslant}
\renewcommand{\geq}{\geqslant}

\newcommand{\Proj}{{\operatorname{Proj}}}

\newcommand{\Cstar}{\operatorname{C^*}} 
\newcommand{\Cstarred}{\operatorname{C_r^*}} 

\newcommand{\eps}{\varepsilon} 

\definecolor{vadimgreen}{rgb}{0,0.7,0.3}

\title{Uniqueness questions for $\Cstar$-norms on group rings}

\author{Vadim Alekseev}
\address{Vadim Alekseev, Technische Universit\"{a}t Dresden, Fakult\"{a}t Mathematik, Institut f\"{u}r Geometrie, 01062 Dresden, Germany}
\email{vadim.alekseev@tu-dresden.de}

\author{David Kyed}
\address{David Kyed, Department of Mathematics and Computer Science, University of Southern Denmark, Campusvej 55, 5230 Odense M, Denmark}
\email{dkyed@imada.sdu.dk}

\begin{document}
\onehalfspace
\maketitle
\begin{abstract}
We provide a large class of discrete amenable groups for which the complex group ring has several $\Cstar$-completions, thus providing partial evidence towards a positive answer to a question raised by Rostislav Grigorchuk, Magdalena Musat and Mikael R{\o}rdam.

\end{abstract}

\section{Introduction}
The interplay between group theory and operator algebras dates back to the seminal papers by Murray and von Neumann \cite{rings-of-operators} and by choosing different completions of a discrete countable group $\Gamma$ one obtains interesting analytic objects; for instance the Banach algebra $\ell^1(\Gamma)$, the full and reduced $\Cstar$-algebras $\Cstar(\Gamma)$ and $\Cstarred(\Gamma)$, and the group von Neumann algebra $L\Gamma$. In general there are many norms on, say, $\ell^1(\Gamma)$ such that the completion with respect to this norm gives a $\Cstar$-algebra, and the question of when the $\Cstar$-completion is unique (in which case $\Gamma$ is said to be \emph{$\Cstar$-unique}) has been studied by various authors  \cite{Ng-Permanence-properties, Boidol, Barnes}.   A $\Cstar$-unique discrete group is evidently amenable and it is, to the best of the authors' knowledge, an open question whether the converse is true, although it is known to be false in the more general context of locally compact groups \cite{Ng-Permanence-properties}. More recently, the paper \cite{GMR} put emphasis on the question of when the complex group algebra $\CC\Gamma$ has a unique $\Cstar$-completion. As is easily seen \cite[Proposition 6.7]{GMR}, if $\Gamma$ is locally finite (i.e.~if every finitely generated subgroup is finite) then $\CC\Gamma$ has a unique $\Cstar$-completion, and \cite[Question 6.8]{GMR} asks if the converse is true. The present paper provides partial evidence towards a positive answer to this, in that we prove that for the following classes of non-locally finite groups have several $\Cstar$-completions.

\begin{theoremletter}[see Proposition \ref{thm:central-element-of-infinite-order} and Corollary \ref{cor:poly-growth}]\label{thm:A}
The class of countable groups $\Gamma$ for which $\CC\Gamma$ does not have a unique $\Cstar$-norm includes the following:
\begin{itemize}
\item[(i)] Infinite groups of polynomial growth.
\item[(ii)] Torsion free, elementary amenable groups with a non-trivial, finite conjugacy class.
\item[(iii)] Groups with a central element of infinite order.
\end{itemize}
\end{theoremletter}

The key to the proof of of (i) and (ii)  is the so-called {strong Atiyah conjecture} (see Section \ref{sec:atiyah})  which predicts a concrete restriction on the von Neumann dimension of kernels of elements in the complex group algebra under the left regular representation  --- notably these are predicted to be either zero or one if the group in question is torsion free.

\subsection*{Acknowledgements}
The authors would like to thank Mikael R{\o}rdam and Thomas Schick for helpful discussions revolving around the topics of the present paper as well as the anonymous referee for the valuable remarks improving Corollary 2.3 and Proposition 3.3.
DK gratefully acknowledge the financial support from  the Villum Foundation (grant no.~7423) and from the Independent Research Fund Denmark (grant no.~7014-00145B).

\section{Basic results on $\Cstar$-uniqueness}
In what follows, all discrete groups are implicitly assumed to be at most countable. We will use several operator algebras associated to a discrete group $\Gamma$: the maximal $\Cstar$-algebra $\Cstar(\Gamma)$, the reduced $\Cstar$-algebra $\Cstarred(\Gamma)$ and the von Neumann algebra $L\Gamma$. For more information on these, we refer to \cite[\S{}2.5]{brown-ozawa}. We recall that $L\Gamma = (\lambda(\CC\Gamma))'' \subset \mathbb B(\ell^2\Gamma)$ is generated by the left regular representation $\lambda\colon \CC\Gamma\to \BB(\ell^2\Gamma)$ and carries a canonical, faithful, normal trace given by $\tau(x) = \ip{x\delta_e}{\delta_e}$. In what follows, $\tr$ will denote the normalized trace on $\MM_n(\CC)$ while $\Tr$ will denote the non-normalized trace.

We begin by formally introducing the notion of $\Cstar$-uniqueness. In order to avoid a notational conflict with the already existing notions studied in  \cite{Ng-Permanence-properties, Boidol}, we emphasize that we are investigating the uniqueness of $\Cstar$-norms on the complex group algebra in contrast to the $\ell^1$-algebra. 

\begin{defi}
Let $\Gamma$ be a discrete group. $\CC\Gamma$ is said to be:
\begin{itemize}
\item[(i)] $\Cstar$-unique if it carries a unique $\Cstar$-norm;
\item[(ii)] $\Cstarred$-unique if no $\Cstar$-norm on $\CC\Gamma$ is properly majorised by the reduced $\Cstar$-norm.
\end{itemize}
$\Gamma$ is said to be algebraically $\Cstar$- (respectively $\Cstarred$-)unique if $\CC\Gamma$ is $\Cstar$- (respectively $\Cstarred$-)unique.
\end{defi}

Amenable groups are characterized by the property that the maximal and reduced $\Cstar$-algebras coincide, and thus a nonamenable group is never algebraically $\Cstar$-unique; on the other hand, for amenable groups the above notions coincide. 
Note also that the class of {$\Cstar$-simple} groups, which has recently received a lot of attention \cite{BKKO, Boudec}, falls within the class of algebraically $\Cstarred$-unique groups.
As already mentioned in the introduction, algebraic $\Cstar$-uniqueness appeared in the recent paper \cite{GMR} in which the authors observed that locally finite groups have this property and asked if this characterizes the class of locally finite groups. Below we prove a few basic permanence results regarding algebraic $\Cstar$-uniqueness, but before doing so we give an alternative characterization, which is straightforward algebraic adaptation of the similar result for $\ell^1$-algebras  \cite[Proposition 2.4]{Barnes}.

\begin{lem} Let $\Gamma$ be a discrete group. Then
$\CC\Gamma$ is $\Cstar$-unique (respectively $\Cstarred$-unique) if and only if every nontrivial closed, two-sided ideal in $\Cstar(\Gamma)$ (respectively $\Cstarred(\Gamma)$) intersects $\CC\Gamma$ non-trivially.
\end{lem}
\begin{proof}
We give the proof for the statement about algebraic $\Cstar$-uniqueness; the other case is obtained by replacing $\Cstar(\Gamma)$ by $\Cstarred(\Gamma)$ throughout the proof.
Assume that there is a non-trivial ideal $J\trianglelefteqslant \Cstar(\Gamma)$ intersecting $\CC\Gamma$ trivially
 and denote by $q\colon \Cstar(\Gamma) \to \Cstar(\Gamma)/J$ the quotient map. Composing $q$ with the inclusion $\CC\Gamma\hookrightarrow \Cstar(\Gamma)$ yields a faithful representation of $\CC\Gamma$ 
 and it defines a $\Cstar$-norm on it that is properly majorised by the maximal norm by non-triviality of $J$. Conversely, if there is a $\Cstar$-norm on $\CC\Gamma$ which is properly majorised by the norm coming from $\Cstar(\Gamma)$, then $\Cstar(\Gamma)$ surjects onto the corresponding quotient, and the kernel of this surjection is a non-trivial ideal intersecting $\CC\Gamma$ trivially. 
\end{proof}

\begin{cor}\label{cor:direct-products}
Let $\Gamma$ and $\Lambda$ be discrete groups. If $\CC(\Gamma\times \Lambda)$ is $\Cstar$-unique (respectively $\Cstarred$-unique), then so are $\CC\Gamma$ and $\CC\Lambda$.
\end{cor}
\begin{proof}

Let $J\trianglelefteqslant \Cstarred(\Gamma)$ be a non-trivial ideal intersecting $\CC\Gamma$ trivially. Then $J\otimes_{\max} \Cstarred(\Lambda) \trianglelefteqslant \Cstar(\Gamma)\otimes_{\max} \Cstar(\Lambda) = \Cstar(\Gamma\times\Lambda)$ is a non-trivial ideal intersecting $\CC (\Gamma\times\Lambda) = \CC \Gamma\otimes_{\alg} \CC\Lambda$ trivially.
The same proof with $\Cstar$ replaced by $\Cstarred$ and $\otimes_\max$ replaced by $\otimes_\min$ works for the reduced case.
\end{proof}

\begin{prop}\label{thm:central-element-of-infinite-order}
If $\Gamma$ is a discrete group with a central element of infinite order then $\CC\Gamma$ is not $\Cstarred$-unique.
\end{prop}
\begin{proof}
Denote by $Z$ the subgroup in $\Gamma$  generated by a central element of infinite order. Then $\Cstarred(Z)\cong C(S^1)$ and $LZ\cong L^\infty(S^1)$ via the Fourier transform and we denote by $p\in LZ$ the projection corresponding to the characteristic function of the upper half circle $\{e^{i\theta}\mid \theta\in [0,\pi]\}$. Define $\pi:=\lambda_\Gamma p$; i.e.~the left regular representation of $\Gamma$ restricted to the invariant subspace $p\ell^2(\Gamma)$. Choosing a non-zero function $f\in C(S^1)$ supported in the \emph{lower} half circle we obtain a non-zero element  $x\in \Cstarred(Z)\subset \Cstarred(\Gamma)$ with $xp=0$ and hence the norm on $\Cstarred(\Gamma)$ induced by $\pi$ is not the one induced by $\lambda_\Gamma$. We now only need to see that $\pi$ is faithful on $\CC\Gamma$. To this end, consider the trace-preserving conditional expectation $E\colon L\Gamma\to LZ$ \cite[Lemma 1.5.11]{brown-ozawa} and assume that $a\in \CC\Gamma$ is in the kernel of $\pi$. Then  $a^*a$ is also in the kernel of $\pi$ and since $E$ is an $LZ$-bimodule map \cite[Proposition 1.5.7]{brown-ozawa} we get
\[
0=E(\lambda_\Gamma(a^*a)p)=E(\lambda_\Gamma(a^*a))p.
\]
However, $E(\CC\Gamma)\subset \CC Z \simeq \Pol(z,\bar{z})\subset C(S^1)$ and therefore $E(\lambda_\Gamma(a^*a))=0$ and since $E$ is trace-preserving and the trace on $L\Gamma$ is faithful we conclude that $a^*a$, and hence $a$, is zero. 
\end{proof}

\begin{cor}\label{cor:conjecture-for-abelian}
An abelian group is algebraically $\Cstar$-unique if and only if it is locally finite (i.e.~pure torsion).
\end{cor}
\begin{rem}
The result in Corollary \ref{cor:conjecture-for-abelian} was also observed, independently and with different proofs, by Rostislav Grigorchuk, Magdalena Musat and Mikael R{\o}rdam (unpublished).
\end{rem}

\begin{rem}
The class of locally finite groups  has many stability properties --- for instance it is closed under subgroups, quotients and extensions and, moreover, being virtually locally finite is the same as being locally finite. However, verifying these properties for the class of $\Cstar$-unique groups seems to be a much bigger challenge.
\end{rem}

\section{The strong Atiyah conjecture and $\Cstarred$-uniqueness}

\subsection{The strong Atiyah conjecture}\label{sec:atiyah}
The key to our main result is the so-called strong Atiyah conjecture which is briefly described in the following. A good general reference is \cite[Chapter 10]{Luck02} where all of the results below can be found, and to which we also refer for the original references. 
 Let $\Gamma$ be a discrete group and denote by $\frac{1}{|\mathop\mathrm{FIN}(\Gamma)|}\ZZ$ the additive subgroup in $\QQ$ generated by the set
\[
\left\{\frac{1}{|\Lambda|} \ \Big\vert \ \Lambda \leq \Gamma \text{ a finite subgroup} \right\}.
\]
Given a matrix $A\in \MM_n(\CC\Gamma)$ we denote by $L_A\in L\Gamma\otimes \MM_n(\CC)\subset \BB(\ell^2(\Gamma)^n)$ the bounded operator given by left multiplication with $A$ (via the left regular representation of $\Gamma$). The strong Atiyah conjecture then predicts that
\[
\dim_{L\Gamma} \ker(L_A) \coloneqq (\tau\otimes \Tr)(P_{\ker L_A}) \in \frac{1}{|\mathop\mathrm{FIN}(\Gamma)|}\ZZ.
\]
Here $\dim_{L\Gamma}(-)$ denotes the von Neumann dimension of the (right) Hilbert $L\Gamma$-module $\ker(L_A)$ defined as the non-normalized trace of the kernel projection $P_{\ker L_A}$; see \cite{Luck02} for details on this. It should be noted that the strong Atiyah conjecture is  false in general \cite[Theorem 10.23]{Luck02}, but is known to hold for all groups which have a bound on the order of finite subgroups and belong to Linnell's class 
$\C$ \cite[Theorem 10.19]{Luck02}, the latter being the smallest class of groups which contain all free groups, is closed under directed unions and extensions by elementary amenable groups (i.e., if $\Lambda\trianglelefteqslant \Gamma$, $\Lambda \in \C$ and $\Gamma/\Lambda$ is elementary amenable, then $\Gamma\in \C$).
The above discussion motivates  the following notion.
\begin{defi} Let $\Gamma$ be a countable group. The \emph{torsion multiplier} of $\Gamma$ is defined as
\[
\theta(\Gamma) = \frac{1}{\mathop\mathrm{lcm}\{|H|\mid H\leqslant\Gamma\text{ finite\}}}\in [0,1].
\]
\end{defi}
In this definition, and in what follows, we use the convention that the least common multiple (lcm) of an infinite set of natural numbers is infinity and that $\frac{1}{\infty}=0$.
Note that if $\Gamma$ has an upper bound on the set of finite subgroups, then 
\[
 \frac{1}{|\text{FIN(G)}|}\ZZ= \left\{n\theta(\Gamma) \mid n\in \ZZ   \right\},
\]
and $\frac{1}{|\text{FIN(G)}|}\ZZ$ has $0$ as an accumulation point otherwise. In view of this, the strong Atiyah conjecture for a group $\Gamma$ with $\theta(\Gamma) > 0$ implies that the possible kernel dimensions are properly quantized in the sense that they can only take values in the discrete set $\left\{ n\theta(\Gamma) \mid n\in \NN   \right\}\subset \RR$.
Theorem \ref{thm:A} (i) and (ii) will follow directly from our main technical result, Theorem \ref{thm:main} below. The key idea in the proof is to play the aforementioned ``quantization'' of the kernel dimensions against an abundance of central projections in $L\Gamma$ with small traces which provide representations of $\Cstarred(\Gamma)$ with non-trivial kernels. To quantify this, we need the following definition.

\begin{defi}
The \emph{central granularity} of $\Gamma$ is defined as
\[
\sigma(\Gamma) = \inf \{\tau(p)\mid p\in \Proj(Z(L\Gamma)),\,p\neq 0\}\in [0,1].
\]
\end{defi}
We note that $\sigma(\Gamma) < 1$ if and only if $Z(L\Gamma)$ is nontrivial which is equivalent $\Gamma$ not being icc\footnote{recall that a group is \emph{icc} if all non-trivial conjugacy classes are infinite}.
The next proposition computes the central granularity of $\Gamma$ in group-theoretic terms. Recall that the \emph{FC-centre} $\Gamma_{\mathrm{fc}}$ is the normal subgroup of $\Gamma$ consisting of all elements with finite conjugacy classes. 
\begin{prop}\label{prop:granularity}
Let $\Gamma_{\mathrm{fc}} \trianglelefteqslant \Gamma$ be the FC-centre of $\Gamma$. Then
\[
\sigma(\Gamma) = \frac{1}{|\Gamma_{\mathrm{fc}}|},
\]
where the right-hand side is interpreted as $0$ if $|\Gamma_{\mathrm{fc}}| = \infty$.
\end{prop}
\begin{proof}



$\Gamma_{\mathrm{fc}}$ is an increasing union of a sequence of finitely generated normal subgroups $\Lambda_n\trianglelefteqslant \Gamma$; to see this, note that $\Gamma_\mathrm{fc}$ is clearly an increasing union of a sequence of finitely generated subgroups $\Lambda_n'$, and defining $\Lambda_n$ to be generated by the $\Gamma$-conjugacy classes of a finite system of generators for $\Lambda_n'$ yields the desired sequence of finitely generated subgroups which are normal in $\Gamma$. We now have two cases to consider:
\begin{enumerate}
\item[(i)] all $\Lambda_n$ are finite (equivalently, $\Gamma_{\mathrm{fc}}$ is a torsion group),
\item[(ii)] $\Lambda_n$ is infinite for some $n$.
\end{enumerate}
In case (i), setting $p_n:=\frac{1}{|\Lambda_n|}\sum_{g\in \Lambda_n} g$, we get a projection $p_n\in L\Gamma_{\mathrm{fc}}$ with $\tau(p_n)=\frac{1}{|\Lambda_n|}$; moreover, $p_n$ is central in $L\Gamma$ since $\Lambda_n$ is normal in $\Gamma$. This proves that $\sigma(\Gamma) = 0$ if $\Gamma_{\mathrm{fc}}$ is an infinite torsion group (in this case $|\Lambda_n|\to \infty$). If $\Gamma_{\mathrm{fc}}$ is finite, then the sequence stabilizes, and therefore we get a central projection $p$ in $L\Gamma$ with trace $\frac{1}{|\Gamma_{\mathrm{fc}}|}$. The centre of $L\Gamma$ consists of elements whose associated Fourier series in $\ell^2(\Gamma)=L^2(L\Gamma, \tau)$ are supported only on $\Gamma_{\mathrm{fc}}$ and are constant along conjugacy classes, and is therefore  contained in the centre of $L\Gamma_{\mathrm{fc}}$; hence we get $\frac{1}{|\Gamma_{\mathrm{fc}}|} \geqslant\sigma(\Gamma) \geqslant \sigma(\Gamma_{\mathrm{fc}})$. But we also have $L\Gamma_{\mathrm{fc}} = \mathbb C\Gamma_{\mathrm{fc}}$ which by representation theory of finite groups is isomorphic to a direct sum of matrix algebras $\bigoplus_\pi \mathbb M_{d_\pi}(\mathbb C)$ with the trace given by $\bigoplus_\pi \frac{d^2_\pi}{|\Gamma_{\mathrm{fc}}|} \tr$; thus, the minimal central projection has trace $\frac{1}{|\Gamma_{\mathrm{fc}}|}=\sigma(\Gamma_{\mathrm{fc}})$; this proves the claim.

In case (ii) we fix an $n\in\NN$ such that $\Lambda_n\eqqcolon\Lambda$ is infinite and note that since $\Lambda$ is generated by a finite number of elements with finite conjugacy classes, its centralizer $C_\Gamma(\Lambda)$ is of finite index in $\Gamma$.
We now claim that $L\Lambda$ has a diffuse von Neumann subalgebra and thus projections of arbitrarily small trace. This can be seen as follows: if $L\Lambda$ has a direct summand of type $\mathrm{II}_1$, it is clear because such von Neumann algebras are diffuse. Otherwise $L\Lambda$ is of type $\mathrm{I}$, but then $\Lambda$ is virtually abelian \cite[Lemma 3.3]{Luck97}, and hence, being infinite by assumption and finitely generated by construction, contains 
a copy of $\ZZ$ which generates a diffuse von Neumann algebra $L\ZZ\cong L^\infty(S^1)$.
In view of the above, for an arbitrary $\eps > 0$ there is a projection $p \in L\Lambda\subset L\Gamma_{\mathrm{fc}}$ of trace $\tau(p) < \frac{\eps}{[\Gamma:C_\Gamma(\Lambda)]}$.  Now let
\[
q\coloneqq \bigvee_{g\in \Gamma} \prescript{g}{}p,
\]  
where $ \prescript{g}{}p :=gpg^{-1}$.  Then $q$ is a central projection in $L\Gamma$. Moreover, $p$ is invariant under the centralizer $C_\Gamma(\Lambda)$ and upon choosing coset representatives $g_1,\dots, g_{[\Gamma:C_\Gamma(\Lambda)]}$ for $\Gamma/C_\Gamma(\Lambda)$ we obtain that
 \[
q = \bigvee_{i=1}^{[\Gamma:C_\Gamma(\Lambda)]} \prescript{g_i}{}p  
\]
and hence $\tau(q) \leqslant [\Gamma:C_\Gamma(\Lambda)]\cdot \tau(p) < \eps$. Thus $\sigma(\Gamma) = 0$.
\end{proof}


\begin{lemma}\label{lemma:small-support}
Let $\Gamma$ be a  discrete non-icc group. For every $\eps > 0$ there exists a nonzero projection $p\in Z(L\Gamma)$ with $\tau(p)< \sigma(\Gamma) + \eps$ and a nonzero, central element $x\in \Cstarred(\Gamma)$ with $xp^\perp=0$. 
\end{lemma}
\begin{proof}
Since $\Gamma$ is non-icc, $Z(L\Gamma)\neq \CC1$ so $\sigma(\Gamma)<1$. Let $\varps>0$ be given and assume, without loss of generality, that $\sigma(\Gamma)+\varps<1$.
One has $Z(L\Gamma)=Z(\Cstarred(\Gamma))''=Z(\CC\Gamma)''$, as can bee seen for instance by using Kaplansky's density theorem together with the center valued trace, and noting that $Z(\CC\Gamma)$ consists of the elements whose coefficients are constant along conjugacy classes.
By Gelfand duality, $Z(\Cstarred(\Gamma))$ is isomorphic to the $\Cstar$-algebra $C(Z)$ of continuous functions on its Gelfand spectrum $Z$, which is a compact Hausdorff space; it is metrizable because $\Cstarred(\Gamma)$ is separable. The canonical trace $\tau$ thus gives a regular Borel probability measure $\mu$ on $Z$ \cite[Theorem 2.14]{rudin-real-and-complex} and an isomorphism $Z(L\Gamma)=Z(\Cstarred(\Gamma))''\cong L^\infty(Z,\mu)$ compatible with the natural inclusions. Projections in $Z(L\Gamma)$ correspond via this isomorphism to measurable subsets of $Z$ (up to null sets), and we therefore obtain a measurable subset $A\subset Z$ such that $0 < \mu(A)<\sigma(\Gamma)+\varps/2$.  By regularity of $\mu$, there exists $U\supseteq A$ open such that 
\[
0 < \mu (A) \leqslant \mu(U) < \sigma(\Gamma) + \eps < 1.
\]
Now, there is a non-zero element $x\in C(Z)$ vanishing on the compact set $K\coloneqq Z\setminus U$ (for instance, the distance function to $K$); letting $p$ be the projection corresponding to $U$ finishes the proof.
\end{proof}

The following lemma gives a concrete description of the decomposition of the left regular representation of a discrete group $\Gamma$ over the cosets of a finite index normal subgroup $\Lambda$. 

\begin{lemma}\label{lemma:coset-decomposition}
Let $\Lambda\trianglelefteqslant \Gamma$ be a normal subgroup of finite index. 
For every choice of coset representatives $g_1,\dots, g_{[\Gamma: \Lambda]}\in \Gamma$ there 
exists a trace-preserving inclusion of von Neumann algebras $\pi\colon (L\Gamma,\tau)\hookrightarrow (\mathbb M_{[\Gamma : \Lambda]}(L\Lambda),\tau\otimes\mathrm{tr})$ which restricts to corresponding inclusions at the level of reduced $\Cstar$-algebras and complex group rings, and which for $x\in L\Lambda$ is given by
\begin{equation}\label{eq:diagonal-elements}
\pi(x) = \mathop\mathrm{diag}( \prescript{g_1}{} x,\prescript{g_2}{} x,\dots, \prescript{g_{[\Gamma:\Lambda]}}{} x), 
\end{equation}
where $\prescript{g}{} x = gxg^{-1}$ is the conjugation action of $g\in\Gamma$ on $L\Lambda$. 
\end{lemma}
\begin{proof}
Choose coset representatives $g_1,g_2,\dots,g_{[\Gamma:\Lambda]}$ of $\Gamma/\Lambda$ and consider the isomorphisms of Hilbert spaces
\[
\ell^2(\Gamma) \cong \bigoplus_{i=1}^{[\Gamma:\Lambda]} \ell^2(g_i^{-1}\Lambda) \cong \bigoplus_{i=1}^{[\Gamma:\Lambda]} \ell^2(\Lambda).
\]
These induce a $\ast$-isomorphism $\pi\colon \mathbb B(\ell^2\Gamma) \xrightarrow{\cong} \mathbb M_n(\BB(\ell^2(\Lambda)))$. It is routine to check that $\pi$ restricts to a trace-preserving inclusion of $\CC\Gamma$ into $\MM_{[\Gamma :\Lambda]}(\CC\Lambda)$ which automatically implies the corresponding results for the reduced $\Cstar$-algebras and von Neumann algebras. Finally, for $h\in \Lambda$ we have
\[
\pi(h) = \mathop\mathrm{diag}(\lambda(h), \dots, \lambda(h)) \in \bigoplus_{i=1}^{[\Gamma:\Lambda]} \mathbb B(\ell^2(g_i^{-1}\Lambda)),
\]
and thus formula \eqref{eq:diagonal-elements} follows in view of the identity 
\[
h g_i^{-1} h' = g_i^{-1} (\prescript{g_i}{} h) h',\quad h,h'\in\Lambda.\qedhere
\]
\end{proof}

\begin{thm}\label{thm:main}
Let $\Lambda$ be a discrete group satisfying the strong Atiyah conjecture and let $\Lambda\trianglelefteqslant \Gamma$ be a finite index inclusion of $\Lambda$ into a group $\Gamma$ as a normal subgroup. If $[\Gamma:\Lambda]^2\cdot \sigma(\Lambda) < \theta(\Lambda)$ then $\CC\Gamma$ is not $\Cstarred$-unique.
\end{thm}
\begin{proof}
The assumption $[\Gamma:\Lambda]^2\cdot \sigma(\Lambda) < \theta(\Lambda)$ forces $\Lambda$ to be non-icc and applying
Lemma \ref{lemma:small-support} we get a projection $p\in Z(L\Lambda)$ with $\tau(p) < \frac{\theta(\Lambda)}{[\Gamma:\Lambda]^2}$ and a non-zero central element $x\in \Cstarred(\Lambda)$ with  $xp^\perp=0$. We are going to construct a representation of $\Cstarred(\Gamma)$ which is injective on $\CC\Gamma$ but with $x$ in the kernel. 
To this end, consider a set of coset representatives $g_1,\dots, g_{[\Gamma:\Lambda]}$ for $\Gamma/\Lambda$ and the $*$-homomorphism $\pi\colon L\Gamma \to \MM_{[\Gamma:\Lambda]}(L\Lambda)$ provided by Lemma \ref{lemma:coset-decomposition}. 
From this we obtain a central projection $q := \bigvee_{i=1}^{[\Gamma:\Lambda]}  \prescript{g_i}{} p  \in Z(L\Lambda)$, 
and  cutting $\pi$ with the complement of $ \tilde{q}:= \mathop\mathrm{diag}(q, \dots, q)\in Z(\MM_{[\Gamma : \Lambda]}(L\Lambda))$, we get a representation
\[
\pi_q\colon \Cstarred(\Gamma) \to \mathbb B \left(\ell^2(\Lambda)^{[\Gamma:\Lambda]}\tilde{q}^\perp\right),
\]
\[
a\mapsto \pi(a) \tilde{q}^\perp.
\]
As $q^\perp = \bigwedge_{i=1}^{[\Gamma:\Lambda]}  \prescript{g_i}{} (p^\perp) $ and $xp^\perp=0$, it follows that $x\in \ker\pi_q$ in view of \eqref{eq:diagonal-elements}. Let $a\in \CC\Gamma\cap \ker \pi_q$. This means that $\pi(a)\tilde{q}^\perp = 0$, and thus the kernel projection $r$ of $\pi(a)$ satisfies  $r \geqslant \tilde{q}^\perp$. Therefore
\[
(\tau\otimes\mathrm{Tr})(r)\geqslant (\tau\otimes\mathrm{Tr})(\tilde{q}^\perp) \geq  [\Gamma:\Lambda](1 - [\Gamma:\Lambda]\tau(p)) > [\Gamma:\Lambda] - \theta(\Lambda).
\] 
On the other hand, the assumption  $[\Gamma:\Lambda]^2\cdot \sigma(\Lambda) < \theta(\Lambda)$ forces an upper bound on the order of finite subgroups in $\Lambda$, i.e.~$\theta(\Lambda)>0$, and since $\Lambda$ is furthermore assumed to satisfy the strong Atiyah conjecture we obtain (using the notation of Section \ref{sec:atiyah}) that
\[
\dim_{L\Lambda}(\ker (L_A)) = (\tau\otimes\mathrm{Tr})(P_{\ker L_A}) \in \{n\theta(\Lambda) \mid n\in \ZZ\}
\]
for any matrix $A\in \MM_{[\Gamma :\Lambda]}(\CC\Lambda)$. Thus $(\tau\otimes\mathrm{Tr})(r)\leqslant [\Gamma:\Lambda] - \theta(\Lambda)$ unless $\pi(a) = 0$. This proves that $\pi_q$ is injective on $\CC\Gamma$ and hence completes the proof.
\end{proof}

As a corollary, we deduce that some important families of  groups are not $\Cstarred$-unique. In particular this includes the groups mentioned in Theorem \ref{thm:A} (i) and (ii),  and together with Proposition \ref{thm:central-element-of-infinite-order} this completes the proof of Theorem \ref{thm:A}.
\begin{cor}\label{cor:poly-growth}
All groups in following classes are not $\Cstarred$-unique:
\begin{itemize}
\item[(i)] Torsion free, non-icc groups satisfying the strong Atiyah  conjecture; in particular all elementary
amenable, non-icc, torsion free groups.

\item[(ii)] Virtually polycyclic groups with infinite FC-centre; in particular, all infinite groups of polynomial growth.
\end{itemize}
\end{cor}
\begin{proof}
To see (i), note that the existence of a non-trivial finite conjugacy class implies the existence of a non-trivial central element in $\CC\Gamma$ (namely the sum of the elements in the finite conjugacy class) and hence a non-trivial projection in $Z(L\Gamma)$; thus $\sigma(\Gamma)<1$. Moreover, since $\Gamma$ is torsion free, $\theta(\Gamma)=1$ and since $\Gamma$ is assumed to satisfy the strong Atiyah conjecture it follows $\Cstarred$-unique by Theorem \ref{thm:main}. The last statement in (i) follows directly from this since the elementary amenable groups are contained in Linnell's class $\C$ (see Section \ref{sec:atiyah}) for which the strong Atiyah conjecture is known to hold in the presence of a bound on the order of finite subgroups \cite[Theorem 10.19]{Luck02}.\\
To see (ii), let $\Lambda\trianglelefteqslant \Gamma$ be a normal finite index polycyclic subgroup of $\Gamma$.
As $\Gamma$ has infinite FC-centre, so does $\Lambda$ and the FC-centre of $\Lambda$ is moreover finitely generated by polycyclicity. 
A classical result by Hirsch \cite[Theorem 3.21]{hirsch-soluble-iii}
 implies that the orders of finite subgroups of $\Lambda$ are bounded; thus $\theta(\Lambda) > 0$. On the other hand, $\sigma(\Lambda) = 0$ by Proposition \ref{prop:granularity} (ii). Moreover, polycyclic groups, being elementary amenable, satisfy the strong Atiyah conjecture. Thus, $\CC\Gamma$ follows non-$\Cstarred$ unique by Theorem \ref{thm:main}. 

Finally, the claim about infinite groups of polynomial growth follows by first observing that by Gromov's theorem \cite{Gromov-poly-growth} these are exactly finitely generated virtually nilpotent groups. As finitely generated nilpotent groups are polycyclic, the claim follows once we argue that virtually nilpotent groups automatically have infinite FC-centre. To see this, recall that a finitely generated virtually nilpotent group $\Gamma$ contains a finite index torsion free nilpotent normal subgroup $\Lambda$ (by polycyclicity and \cite[Theorem 3.21]{hirsch-soluble-iii}). Now it follows that the centre of $\Lambda$ is infinite, and therefore so is the FC-centre $\Lambda_{\mathrm{fc}}$; but as $\Lambda\trianglelefteqslant \Gamma$ is a finite index inclusion, $\Lambda_{\mathrm{fc}}\subseteq \Gamma_{\mathrm{fc}}$. Thus, $\Gamma_{\mathrm{fc}}$ is infinite.
\end{proof}


\begin{thebibliography}{BKKO17}

\bibitem[Bar83]{Barnes}
Bruce~A. Barnes.
\newblock The properties {$^{\ast} $}-regularity and uniqueness of {$C^{\ast}
  $}-norm in a general {$^{\ast} $}-algebra.
\newblock {\em Trans. Amer. Math. Soc.}, 279(2):841--859, 1983.

\bibitem[BKKO17]{BKKO}
Emmanuel Breuillard, Mehrdad Kalantar, Matthew Kennedy, and Narutaka Ozawa.
\newblock {$C^*$}-simplicity and the unique trace property for discrete groups.
\newblock {\em Publ. Math. Inst. Hautes \'Etudes Sci.}, 126:35--71, 2017.

\bibitem[BO08]{brown-ozawa}
Nathanial~P. Brown and Narutaka Ozawa.
\newblock {\em {$C^*$}-algebras and finite-dimensional approximations},
  volume~88 of {\em Graduate Studies in Mathematics}.
\newblock American Mathematical Society, Providence, RI, 2008.

\bibitem[Boi84]{Boidol}
Joachim Boidol.
\newblock Group algebras with a unique {$C^{\ast} $}-norm.
\newblock {\em J. Funct. Anal.}, 56(2):220--232, 1984.

\bibitem[GMR18]{GMR}
Rostislav Grigorchuk, Magdalena Musat, and Mikael R{\o}rdam.
\newblock Just-infinite {$C^*$}-algebras.
\newblock {\em Comment. Math. Helv.}, 93(1):157--201, 2018.

\bibitem[Gro81]{Gromov-poly-growth}
Mikhael Gromov.
\newblock Groups of polynomial growth and expanding maps.
\newblock {\em Inst. Hautes \'Etudes Sci. Publ. Math.}, (53):53--73, 1981.

\bibitem[Hir46]{hirsch-soluble-iii}
Kurt A.~Hirsch.
\newblock On infinite soluble groups. {III}.
\newblock {\em Proc. London Math. Soc. (2)}, 49:184--194, 1946.

\bibitem[LB17]{Boudec}
Adrien Le~Boudec.
\newblock {$C^*$}-simplicity and the amenable radical.
\newblock {\em Invent. Math.}, 209(1):159--174, 2017.

\bibitem[LN04]{Ng-Permanence-properties}
Chi-Wai Leung and Chi-Keung Ng.
\newblock Some permanence properties of {$C^*$}-unique groups.
\newblock {\em J. Funct. Anal.}, 210(2):376--390, 2004.

\bibitem[L{\"u}c97]{Luck97}
Wolfgang L{\"u}ck.
\newblock Hilbert modules and modules over finite von {N}eumann algebras and
  applications to {$L\sp 2$}-invariants.
\newblock {\em Math. Ann.}, 309(2):247--285, 1997.

\bibitem[L{\"u}c02]{Luck02}
Wolfgang L{\"u}ck.
\newblock {\em {$L\sp 2$}-invariants: theory and applications to geometry and
  {$K$}-theory}, volume~44 of {\em Ergebnisse der Mathematik und ihrer
  Grenzgebiete. 3. Folge. A Series of Modern Surveys in Mathematics}.
\newblock Springer-Verlag, Berlin, 2002.

\bibitem[MvN36]{rings-of-operators}
Francis~J. Murray and John~von Neumann.
\newblock On rings of operators.
\newblock {\em Ann. of Math. (2)}, 37(1):116--229, 1936.

\bibitem[Rud66]{rudin-real-and-complex}
Walter Rudin.
\newblock {\em Real and complex analysis}.
\newblock McGraw-Hill Book Co., New York, 1966.

\end{thebibliography}

\def\cprime{$'$} \def\cprime{$'$}

\end{document}